\documentclass[12pt]{amsart}
\usepackage[margin=1.2in]{geometry}
\usepackage{graphicx,latexsym}
\usepackage{comment}
\usepackage{tikz}
\usepackage{amsfonts, amssymb, amsmath, amsthm, bm, hyperref}

\newcommand{\N}{\mathbb{N}}

\newcommand{\R}{\mathbb{R}}

\newcommand{\dx}{{\rm d}x }

\newcommand{\dt}{{\rm d}t }
\newcommand{\dxi}{{\rm d}\xi }

\newtheorem{theorem}{Theorem}
\newtheorem{proposition}{Proposition}

\newtheorem{lemma}{Lemma}
\newtheorem{corollary}{Corollary}

\theoremstyle{definition}

\theoremstyle{remark}
\newtheorem{remark}{Remark}

\numberwithin{equation}{section}

\usetikzlibrary{arrows,chains,matrix,positioning,scopes}
\makeatletter
\tikzset{join/.code=\tikzset{after node path={%
\ifx\tikzchainprevious\pgfutil@empty\else(\tikzchainprevious)%
edge[every join]#1(\tikzchaincurrent)\fi}}}
\makeatother
\tikzset{>=stealth',every on chain/.append style={join},
         every join/.style={->}}
\tikzstyle{labeled}=[execute at begin node=$\scriptstyle,
   execute at end node=$]

\begin{document}
\title[Projective description of Gelfand-Shilov spaces of Roumieu type]{A projective description of generalized Gelfand-Shilov spaces of Roumieu type}%. Topological properties and convolution

\author[A. Debrouwere]{Andreas Debrouwere}
\address{Department of Mathematics, Louisiana State University, Baton Rouge, LA, U.S.A.}
\email{adebrouwere@lsu.edu}
\thanks{A. Debrouwere gratefully acknowledges support by FWO-Vlaanderen, through the postdoctoral grant 12T0519N}

\author[J. Vindas]{Jasson Vindas}
\thanks{The work of J. Vindas was supported by Ghent University, through the BOF-grants 01N01014 and 01J04017.}
\address{Department of Mathematics: Analysis, Logic and Discrete Mathematics, Ghent University, Krijgslaan 281, 9000 Gent, Belgium}
\email{jasson.vindas@UGent.be}

\subjclass[2010]{46E10, 81S30}
\keywords{Gelfand-Shilov spaces; Projective description; Short-time Fourier transform; Ultradifferentiable functions of Roumieu type}
\begin{abstract}
We provide a projective description for a class of generalized Gelfand-Shilov spaces of Roumieu type. 
 In particular, our results apply to the classical Gelfand-Shilov spaces and weighted $L^\infty$-spaces of ultradifferentiable functions of Roumieu type.
\end{abstract}
\maketitle
\section{Introduction}
In general, there is no canonical way to find an explicit and useful system of seminorms describing a given inductive limit topology. However, in many concrete cases this is possible. For weighted $(LB)$-spaces of continuous and holomorphic functions, under quite general assumptions, the topology can be described in terms of weighted sup-seminorms. This problem of \emph{projective description} goes back to the pioneer work of Bierstedt, Meise, and Summers \cite{B-M-S} and plays an important role in Ehrenpreis' theory of analytically uniform spaces \cite{Ehrenpreis,B-D}. On the other hand,  an explicit system of seminorms describing the topology of the space of ultradifferentiable functions of Roumieu type  was first found by Komatsu \cite{Komatsu3}. His proof was based on a structural theorem for the dual space and the same method was later employed by Pilipovi\'{c} \cite{Pil94} to obtain projective descriptions of Gelfand-Shilov spaces of Roumieu type. Such projective descriptions are indispensable for achieving topological tensor product representations of various important classes of vector-valued ultradifferentiable functions of Roumieu type \cite{Komatsu3,P-P-V,D-VCousin}.

The aim of this article is to provide a projective description of a general class of Gelfand-Shilov spaces of Roumieu type. More precisely,  let $(M_p)_{p \in \N}$ be a sequence of positive real numbers and let $\mathcal{V} = (v_n)_{n \in \N}$ be a pointwise decreasing sequence of positive continuous functions on $\R^d$. We study here the $(LB)$-space $\mathcal{B}^{\{M_p\}}_{\mathcal{V}}(\R^d)$ consisting of all those $\varphi \in C^\infty(\R^d)$ such that
$$
 \sup_{\alpha \in \N^d} \sup_{x \in \R^d} \frac{h^{|\alpha|}|\partial^\alpha\varphi(x)|v_n(x)}{M_{|\alpha|}} < \infty
$$ 
for some $h > 0$ and $n \in \N$. Under rather general assumptions on $M_p$ and $\mathcal{V}$, we shall give a projective description of the space $\mathcal{B}^{\{M_p\}}_{\mathcal{V}}(\R^d)$ in terms of Komatsu's family $\mathfrak{R}$ \cite{Komatsu3} and the maximal Nachbin family associated with $\mathcal{V}$ \cite{B-M-S}. We mention that we have already studied the problem in \cite[Prop. 4.16]{D-VWeightedind17}, but we will present here a new approach. Our arguments are  based on the mapping properties of the \emph{short-time Fourier transform} on $\mathcal{B}^{\{M_p\}}_{\mathcal{V}}(\R^d)$ and the projective description of weighted $(LB)$-spaces of continuous functions. We believe this is a transparent and flexible method. It has the advantage that one can work under very mild conditions on $M_p$ and $\mathcal{V}$ and that it avoids duality theory; in fact, our result can be employed to more easily study dual spaces, e.g.\ one might easily deduce structural theorems from it without resorting to a rather complicated dual Mittag-Leffler argument. 

Our general references are \cite{B-M-S} for weighted inductive limits of spaces of continuous functions, \cite{P-P-V} for Gelfand-Shilov spaces, and \cite{Grochenig} for the short-time Fourier transform.

\section{Weighted inductive limits of spaces of continuous functions}\label{sect-weighted-ind-limits}
In this section we recall a result of Bastin \cite{Bastin} concerning the projective description of weighted $(LB)$-spaces of continuous functions. This result will play a key role in the proof of our main theorem.

Let $X$ be a completely regular Hausdorff space. Given a non-negative function $v$ on $X$ we write $Cv(X)$ for the seminormed space consisting of all $f \in C(X)$ such that $\|f\|_v := \sup_{x \in X} |f(x)|v(x) < \infty$. If $v$ is positive, then $\| \, \cdot \, \|_v$ is actually a norm and if, in addition, $1/v$ is locally bounded, then $Cv(X)$ is complete and hence a Banach space. These requirements are fulfilled if $v$ is positive and continuous. A (pointwise) decreasing sequence $\mathcal{V} = (v_n)_{n \in \N}$ of  positive continuous functions on $X$ is called a \emph{decreasing weight system on $X$}. We define
$$
\mathcal{V}C(X)= \varinjlim_{n \in \N} Cv_n(X),
$$
a Hausdorff $(LB)$-space. The \emph{maximal Nachbin family associated with $\mathcal{V}$}, denoted by $ \overline{V}=\overline{V}(\mathcal{V}) $, is given by the space of all non-negative upper semicontinuous functions $v$ on $X$ such that $\sup_{x \in X} v(x)/v_n(x) < \infty$ for all $n \in \N$. 

The \emph{projective hull of $\mathcal{V}C(X)$}, denoted by $C\overline{V}(X)$, is defined as the space consisting of all $f \in C(X)$ such that  $\|f\|_v < \infty$ for all $v \in \overline{V}$. The space $C\overline{V}(X)$ is endowed with the topology generated by the system of seminorms $\{ \| \, \cdot \, \|_v \, : \, v \in \overline{V} \}$. An elementary argument by contradiction shows that $\mathcal{V}C(X)$ and $C\overline{V}(X)$ coincide algebraically and that these spaces have the same bounded sets (cf.\ \cite[Lemma 4.11]{D-VWeightedind17}). The problem of projective description in this context is to characterize the weight systems $\mathcal{V}$ for which the spaces $\mathcal{V}C(X)$ and $C\overline{V}(X)$ are equal as locally convex spaces. In this regard, there is the following result due to Bastin.

\begin{theorem}[{\cite{Bastin}}]\label{Bastin} Let $\mathcal{V} = (v_n)_{n \in \N}$ be a decreasing weight system on $X$ satisfying condition $(V)$, i.e., for every sequence of positive numbers $(\lambda_n)_{n \in \N}$ there is $v \in \overline{V}$ such that for every $n \in \N$ there is $N \in \N$ such that $\inf\{\lambda_1v_1, \ldots, \lambda_{N}v_{N} \} \leq \sup\{ v_n/n, v\}$. Then, $\mathcal{V}C(X)$ and  $C\overline{V}(X)$ coincide topologically.
\end{theorem}
\begin{remark}

Bastin also showed that if for every $v \in \overline V$ there is a positive continuous $\overline v \in \overline V$ such that $v \leq \overline{v}$, then condition $(V)$ is also necessary for the topological identity $\mathcal{V}C(X) = C\overline{V}(X)$. We mention that if $X$ is a discrete or a locally compact $\sigma$-compact Hausdorff space, then every decreasing weight system $\mathcal{V}$ on $X$ satisfies the above condition \cite[p.\ 112]{B-M-S}. 
\end{remark}

\begin{remark}\label{remark-S}
A decreasing weight system $\mathcal{V} = (v_n)_{n \in \N}$ is said to satisfy condition $(S)$ if for every $n \in \N$ there is $m > n$ such that $v_m/v_n$ vanishes at $\infty$. Every weight system satisfying $(S)$ also satisfies $(V)$, but the latter property also holds for constant weight systems (for which $(S)$ obviously fails).
\end{remark}

Let $X$ and $Y$ be completely regular Hausdorff spaces and let $\mathcal{V} =(v_n)_{n \in \N}$ and $\mathcal{W} =(w_n)_{n \in \N}$ be decreasing weight systems on $X$ and $Y$, respectively. We denote by $\mathcal{V} \otimes \mathcal{W} := (v_n \otimes w_n)_{n \in \N}$ the decreasing weight system on $X \times Y$ given by 
$v_n \otimes w_n(x,y) := v_n(x)w_n(y)$, $x \in X, y \in Y$.
\begin{remark}\label{remark-tensor-weights-1}
Let  $\mathcal{V}=(v_n)_{n \in \N}$ and $\mathcal{W}=(w_n)_{n \in \N}$ be decreasing weight systems on $X$ and $Y$, respectively. If both $\mathcal{V}$ and $\mathcal{W}$ satisfy $(V)$, then also $\mathcal{V} \otimes \mathcal{W}$ satisfies $(V)$.
\end{remark}

\begin{remark}\label{remark-tensor-weights-2}
Let $\mathcal{V}=(v_n)_{n \in \N}$ and $\mathcal{W}=(w_n)_{n \in \N}$ be decreasing weight systems on $X$ and $Y$, respectively. Then, for every $u \in \overline{V}(\mathcal{V} \otimes \mathcal{W})$ there are $v \in \overline{V}(\mathcal{V})$ and $w \in \overline{V}(\mathcal{W})$ such that $u \leq v \otimes w$.
\end{remark}
\section{Generalized Gelfand-Shilov spaces of Roumieu type}
We now introduce the class of Gelfand-Shilov spaces of Roumieu type that we are interested in. They are defined via a weight sequence and  a decreasing weight system $\mathcal{V}$ (on $\R^d$) and our aim is to give a projective description of these spaces in terms of Komatsu's family $\mathfrak{R}$ (defined below) and the maximal Nachbin family associated with $\mathcal{V}$.

Let $(M_p)_{p \in \N}$ be a \emph{weight sequence}, that is, a positive sequence that satisfies the condition $\lim_{p \to \infty} M_p / M_{p -1} = \infty$. For $h > 0$ and a non-negative function $v$ on $\R^d$ we write $\mathcal{D}^{M_p,h}_{L^\infty_v}(\R^d)$ for the seminormed space consisting of all $\varphi \in C^\infty(\R^d)$ such that
$$
\| \varphi \|_{\mathcal{D}^{M_p,h}_{L^\infty_v}} := \sup_{\alpha \in \N^d} \sup_{x \in \R^d} \frac{h^{|\alpha|}|\partial^\alpha\varphi(x)|v(x)}{M_\alpha} < \infty,
$$ 
where we write $M_\alpha = M_{|\alpha|}$, $\alpha \in \N^d$. Following Komatsu \cite{Komatsu3}, we denote by $\mathfrak{R}$ the set of all positive increasing sequences $(r_j)_{j \in \N}$ tending to infinity. For $r_j \in \mathfrak{R}$ and a non-negative function $v$ on $\R^d$ we write $\mathcal{D}^{M_p,r_j}_{L^\infty_v}(\R^d)$ for the seminormed space consisting of all $\varphi \in C^\infty(\R^d)$ such that
$$
\| \varphi \|_{\mathcal{D}^{M_p,r_j}_{L^\infty_v}} := \sup_{\alpha \in \N^d} \sup_{x \in \R^d} \frac{|\partial^\alpha\varphi(x)|v(x)}{M_\alpha\prod^{|\alpha|}_{j=0} r_j} < \infty.
$$ 
Similarly as before, $\| \, \cdot \, \|_{\mathcal{D}^{M_p,h}_{L^\infty_v}}$ and $\| \, \cdot \, \|_{\mathcal{D}^{M_p,r_j}_{L^\infty_v}}$ are actually norms if $v$ is positive while $\mathcal{D}^{M_p,h}_{L^\infty_v}(\R^d)$ and $\mathcal{D}^{M_p,r_j}_{L^\infty_v}(\R^d)$ are complete and thus Banach spaces if additionally $1/v$ is locally bounded. Given a decreasing weight system $\mathcal{V} = (v_n)_{n \in \N}$, we define, as in the introduction,
$$
\mathcal{B}^{\{M_p\}}_\mathcal{V}(\R^d) := \varinjlim_{n \to \infty} \mathcal{D}^{M_p,1/n}_{L^\infty_{v_n}}(\R^d),
$$
a Hausdorff $(LB)$-space. Moreover, we write $\widetilde{\mathcal{B}}^{\{M_p\}}_\mathcal{V}(\R^d)$ for the space consisting of all $\varphi \in C^\infty(\R^d)$ such that $\| \varphi \|_{\mathcal{D}^{M_p,r_j}_{L^\infty_v}} < \infty$ for all $r_j \in \mathfrak{R}$ and $v \in \overline{V}$, and endow it with the topology generated by the system of seminorms $\{ \| \, \cdot \, \|_{\mathcal{D}^{M_p,r_j}_{L^\infty_v}} \, : \,  r_j \in \mathfrak{R}, v \in \overline{V} \}$.

\begin{lemma}\label{equality-sets}
Let $M_p$ be a weight sequence and let $\mathcal{V} = (v_n)_{n \in \N}$ be a decreasing weight system. Then, $\mathcal{B}^{\{M_p\}}_\mathcal{V}(\R^d)$ and $\widetilde{\mathcal{B}}^{\{M_p\}}_\mathcal{V}(\R^d)$ coincide algebraically and the inclusion mapping $\mathcal{B}^{\{M_p\}}_\mathcal{V}(\R^d) \rightarrow \widetilde{\mathcal{B}}^{\{M_p\}}_\mathcal{V}(\R^d)$ is continuous.
\end{lemma}
\begin{proof}
It is obvious that $\mathcal{B}^{\{M_p\}}_\mathcal{V}(\R^d)$ is continuously included in $\widetilde{\mathcal{B}}^{\{M_p\}}_\mathcal{V}(\R^d)$. For the converse inclusion, we consider the  decreasing weight system $\mathcal{W} = (w_n)_{n \in \N}$ on $\N^d$ (endowed with the discrete topology), given by $w_n(\alpha) := n^{-|\alpha|}$, $\alpha \in \N^d$. Now let $\varphi \in \widetilde{\mathcal{B}}^{\{M_p\}}_\mathcal{V}(\R^d)$ be arbitrary and define $f(x,\alpha) =\partial^\alpha \varphi(x)/M_\alpha$ for $x \in \R^d$, $\alpha \in \N^d$. By Remark \ref{remark-tensor-weights-2} and  \cite[Lemma 3.4$(ii)$]{Komatsu3} we have that $f \in C\overline{V}(\mathcal{V} \otimes \mathcal{W})(\R^d \times \N^d)$. Since $ \mathcal{V} \otimes \mathcal{W}C(\R^d \times \N^d) =  C\overline{V}(\mathcal{V} \otimes \mathcal{W})(\R^d \times \N^d)$ as sets (cf.\ Section \ref{sect-weighted-ind-limits}), we obtain that $f \in  \mathcal{V} \otimes \mathcal{W}C(\R^d \times \N^d)$, which precisely means that $\varphi \in \mathcal{B}^{\{M_p\}}_\mathcal{V}(\R^d)$.
\end{proof}

The rest of this article is devoted to showing that, under mild conditions on $M_p$ and $\mathcal{V}$, the equality $\mathcal{B}^{\{M_p\}}_\mathcal{V}(\R^d) = \widetilde{\mathcal{B}}^{\{M_p\}}_\mathcal{V}(\R^d)$ also holds topologically. 

We will make use of the following two standard conditions for weight sequences: $(M.1)$ $M_p^2 \leq M_{p-1}M_{p+1}$, $p \geq 1$, and  $(M.2)'$ $M_{p+1} \leq C_0H^{p}M_p$, $p\in \N$, for some $C_0,H \geq 1$. We also need the \emph{associated function}  of the sequence $M_p$ in our considerations, which is given by
$$
M(t):=\sup_{p\in\N}\log\frac{t^pM_0}{M_p},\qquad t > 0,
$$
and $M(0):=0$. We define $M$ on $\R^d$ as the radial function $M(x) = M(|x|)$, $x \in \R^d$. The assumption $(M.2)'$ implies that $M(H^kt) -M(t) \geq k\log(t/C_0)$, $t,k \geq 0$ \cite[Prop.\ 3.4]{Komatsu}. In particular, we have that $e^{M(t)-M(H^{d+1}t)} \leq (2C_0)^{d+1}(1+t^{d+1})^{-1}$, $t \geq 0$. Given $r_j \in \mathfrak{R}$, we denote by $M_{r_j}$ the associated function of the weight sequence $M_p \prod^p_{j=0}r_j$. 

As mentioned in the introduction, our arguments will rely on the mapping properties of the short-time Fourier transform, which we now introduce. The translation and modulation operators are denoted by $T_xf = f(\:\cdot\: - x)$ and $M_\xi f = e^{2\pi i \xi \cdot} f$, for $x, \xi \in \R^d$. The \emph{short-time Fourier transform (STFT)} of a function $f \in L^2(\R^d)$ with respect to a window function $\psi \in L^2(\R^d)$ is defined as
$$
V_\psi f(x,\xi) := (f, M_\xi T_x\psi)_{L^2} = \int_{\R^d} f(t) \overline{\psi(t-x)}e^{-2\pi i \xi t} \dt, \qquad (x, \xi) \in \R^{2d}.
$$
We have that $\|V_\psi f\|_{L^2(\R^{2d})} = \|\psi\|_{L^2}\|f\|_{L^2}$. In particular, the mapping $V_\psi : L^2(\R^d) \rightarrow L^2(\R^{2d})$ is continuous. The adjoint of $V_\psi$ is given by the weak integral
$$
V^\ast_\psi F = \int \int_{\R^{2d}} F(x,\xi) M_\xi T_x\psi \dx \dxi, \qquad F \in L^2(\R^{2d}).
$$
If $\psi \neq 0$ and $\gamma \in L^2(\R^d)$ is a synthesis window for $\psi$, that is $(\gamma, \psi)_{L^2} \neq 0$, then
\begin{equation}
\frac{1}{(\gamma, \psi)_{L^2}} V^\ast_\gamma \circ V_\psi = \operatorname{id}_{L^2(\R^d)}.
\label{reconstruction-L2}
\end{equation}

We are interested in the STFT on the spaces $\mathcal{B}^{\{M_p\}}_\mathcal{V}(\R^d)$ and $\widetilde{\mathcal{B}}^{\{M_p\}}_\mathcal{V}(\R^d)$. This requires to impose some further conditions on the weight system $\mathcal{V}$. Let $A_p$ be a weight sequence with associated function $A$. A decreasing weight system $\mathcal{V} = (v_n)_{n  \in \N}$ is said to be \emph{$A_p$-admissible} if there is $\tau > 0$ such that for every $n \in \N$ there are $m \geq n$ and $C > 0$ such that $v_m(x+y) \leq Cv_n(x)e^{A(\tau y)}$, $x,y \in \R^d$. We start with two lemmas. As customary \cite{P-P-V}, given two weight sequences $M_p$ and $A_p$, we denote by $\mathcal{S}^{(M_p)}_{(A_p)}(\R^d)$ the Gelfand-Shilov space of  Beurling type. For the weight function $v = e^{A(\tau \, \cdot \,)}$, $\tau > 0$, we use the alternative notation $\|\:\cdot \:\|_{\mathcal{S}^{M_p,h}_{A_p,\tau}}=\|\:\cdot \:\|_{\mathcal{D}^{M_p,h}_{L^{\infty}_{v}}}$ so that the Fr\'{e}chet space structure of $\mathcal{S}^{(M_p)}_{(A_p)}(\R^d)$ is determined by the family of norms $\{ \|\:\cdot \:\|_{\mathcal{S}^{M_p,h}_{A_p,\tau}} \, : \, h > 0, \tau > 0\}$.
\begin{lemma}\label{STFT-test}
Let $M_p$ and $A_p$ be weight sequences satisfying $(M.1)$ and $(M.2)'$, let $w$ and $v$ be non-negative measurable functions on $\R^d$ such that
\begin{equation}
v(x+y) \leq Cw(x)e^{A(\tau y)}, \qquad x,y \in \R^d,
\label{v-w-inequality}
\end{equation}
for some $C,\tau > 0$, and let $\psi \in \mathcal{S}^{(M_p)}_{(A_p)}(\R^d)$. Then, the mapping
$
V_\psi: \mathcal{D}^{M_p,h}_{L^\infty_w}(\R^d) \rightarrow Cv\otimes e^{M(\pi h \, \cdot \, /\sqrt{d})}(\R^{2d}_{x,\xi})
$
is well-defined and continuous.
\end{lemma}
\begin{proof}
Let $\varphi \in \mathcal{D}^{M_p,h}_{L^\infty_w}(\R^d)$ be arbitrary. For all $\alpha \in \N^d$ and $(x,\xi) \in \R^{2d}$,
\begin{align*}
&|\xi^{\alpha}V_{\psi}\varphi(x, \xi)|v(x) \\
&\leq C(2\pi)^{-|\alpha|} \sum_{\beta \leq \alpha} \binom{\alpha}{\beta} \int_{\R^d} |\partial^\beta\varphi(t)|w(t) |\partial^{\alpha-\beta}\psi(t-x)|e^{A(\tau(t-x))} \dt \\
&\leq C' \|\varphi\|_{\mathcal{D}^{M_p,h}_{L^\infty_w}} (\pi h)^{-|\alpha|} M_\alpha.
\end{align*}
%where $C' = 2C^{d+1}_0C\|\psi\|_{\mathcal{S}^{M_p,h}_{A_p,H^{d+1}\tau}} \int_{\R^d}(1+|\tau t|^{d+1})^{-1}\dt < \infty.$ 
Hence
$$
|V_{\psi}\varphi(x,\xi)|v(x) \leq C' \|\varphi\|_{\mathcal{D}^{M_p,h}_{L^\infty_w}} \inf_{p \in \N} \frac{M_p}{(\pi h|\xi|/\sqrt{d})^{p}}  =C'M_0  \|\varphi\|_{\mathcal{D}^{M_p,h}_{L^\infty_w}} e^{-M(\pi h\xi/\sqrt{d})}.
$$
\end{proof}
\begin{lemma}\label{double-int-test}
Let $M_p$ and $A_p$ be weight sequences satisfying $(M.1)$ and $(M.2)'$, let $w$ and $v$ be non-negative measurable functions on $\R^d$ satisfying \eqref{v-w-inequality}, and let $\psi \in \mathcal{S}^{(M_p)}_{(A_p)}(\R^d)$. Then, the mapping
$V^\ast_\psi: Cw\otimes e^{M(h \, \cdot \,)}(\R^{2d}_{x,\xi}) \rightarrow \mathcal{D}^{M_p,h/(4H^{d+1}\pi)}_{L^\infty_v}(\R^d)$ is well-defined and continuous.
\end{lemma}
\begin{proof}
Let $F \in Cw\otimes e^{M(h \,\cdot \,)}(\R^{2d}_{x,\xi})$ be arbitrary and set $k = h/(2H^{d+1}\pi)$. For each $\alpha \in \N^d$, (we write $\|\:\cdot \:\|_{Cw\otimes e^{M(h \,\cdot \,)}} =\|\:\cdot \:\|$)
\begin{align*}
&\sup_{t \in \R^d} |\partial^\alpha V^\ast_\psi F(t)|v(t) \\
&\leq C\sum_{\beta \leq \alpha} \binom{\alpha}{\beta}  \sup_{t \in \R^d}\int \int_{\R^{2d}} |F(x,\xi)|w(x) (2\pi|\xi|)^{|\beta|}|\partial^{\alpha-\beta}\psi(t-x)|e^{A(\tau(t-x))} \dx \dxi \\
&\leq CM_0^{-1}\|\psi\|_{\mathcal{S}^{M_p,k}_{A_p,H^{d+1}\tau}}\|F\|\frac{M_\alpha}{(k/2)^{|\alpha|}} \int \int_{\R^{2d}}e^{M(2\pi k \xi)-M(h\xi)} e^{A(\tau x)-A(H^{d+1}\tau x)} \dx \dxi \\
&\leq C' \|F\| \frac{M_\alpha}{(h/(4H^{d+1}\pi))^{|\alpha|}}.
\end{align*}
%where $C' = 4M_0^{-1}C_0^{2(d+1)}C\|\psi\|_{\mathcal{S}^{M_p,k}_{A_p,H^{d+1}\tau}}\int_{\R^d}(1+|2\pi k\xi|^{d+1})^{-1} \dxi \int_{\R^d}(1+|\tau x|^{d+1})^{-1} \dx < \infty$.
\end{proof}

Lemmas \ref{STFT-test} and \ref{double-int-test} yield the following corollary.
\begin{corollary}\label{STFT-GSb}
Let $M_p$ and $A_p$ be weight sequences satisfying $(M.1)$ and $(M.2)'$ and denote by $X$ the Fr\'echet space  consisting of all  $F \in C(\R^{2d})$ such that 
$$
\sup_{(x,\xi) \in \R^{2d}}|F(x,\xi)|e^{A(nx) + M(n\xi)} < \infty
$$
for all $n \in \N$. Let $\psi \in \mathcal{S}^{(M_p)}_{(A_p)}(\R^d)$. Then, the mappings
$V_\psi: \mathcal{S}^{(M_p)}_{(A_p)}(\R^d) \rightarrow X$ and  $V^\ast_\psi: X \rightarrow \mathcal{S}^{(M_p)}_{(A_p)}(\R^d) 
$ are well-defined and continuous.
\end{corollary}

We are now able to establish the mapping properties of the STFT on $\mathcal{B}^{\{M_p\}}_\mathcal{V}(\R^d)$. Given a weight sequence $M_p$ with associated function $M$, we define $\mathcal{V}_{\{M_p\}} := (e^{M(\,\cdot\,/n)})_{n \in \N}$, a decreasing weight system on $\R^d$.

\begin{proposition}\label{STFT-gg} Let $M_p$ and $A_p$ be weight sequences satisfying $(M.1)$ and $(M.2)'$, let $\mathcal{V} = (v_n)_{n \in \N}$ be an $A_p$-admissible decreasing weight system, and let $\psi \in \mathcal{S}^{(M_p)}_{(A_p)}(\R^d)$. Then, the following mappings are continuous: 
$$
V_\psi: \mathcal{B}^{\{M_p\}}_\mathcal{V}(\R^d) \rightarrow \mathcal{V}\otimes\mathcal{V}_{\{M_p\}}C(\R^{2d}_{x,\xi}) 
$$
and
$$ 
V^\ast_\psi: \mathcal{V}\otimes\mathcal{V}_{\{M_p\}}C(\R^{2d}_{x,\xi}) \rightarrow  \mathcal{B}^{\{M_p\}}_\mathcal{V}(\R^d).
$$
Assume that $\mathcal{S}^{(M_p)}_{(A_p)}(\R^d) \neq \{0\}$. If $\psi \neq 0$ and $\gamma \in \mathcal{S}^{(M_p)}_{(A_p)}(\R^d)$ is a synthesis window for $\psi$, the following reconstruction formula holds
\begin{equation}
\frac{1}{(\gamma, \psi)_{L^2}} V^\ast_\gamma \circ V_\psi = \operatorname{id}_{ \mathcal{B}^{\{M_p\}}_\mathcal{V}(\R^d)}.
\label{reconstruction-B}
\end{equation}
\end{proposition}
\begin{proof}
Since $\mathcal{V}$ is $A_p$-admissible, the continuity of $V_\psi$ and $V^\ast_\psi$ follows directly from Lemmas \ref{STFT-test} and \ref{double-int-test}, respectively. We now show \eqref{reconstruction-B}. Let $\varphi \in \mathcal{B}^{\{M_p\}}_\mathcal{V}(\R^d)$ be arbitrary. As $V^\ast_\gamma(V_\psi\varphi)$ and $\varphi$ are both $O(e^{A(\tau \cdot)})$-bounded continuous functions, it suffices to show that 
$$
\int_{\R^d} V^\ast_\gamma(V_\psi\varphi)(t) \chi(t) \dt = (\gamma, \psi)_{L^2} \int_{\R^d}\varphi(t) \chi(t) \dt
$$
for all $\chi \in \mathcal{S}^{(M_p)}_{(A_p)}(\R^d)$. Formula \eqref{reconstruction-L2} implies that
\begin{align*}
\int_{\R^d} V^\ast_\gamma(V_\psi\varphi)(t) \chi(t) \dt &= \int_{\R^d} \left(\int \int_{\R^{2d}} V_\psi \varphi(x,\xi) M_\xi T_x \gamma(t) \dx \dxi \right)\chi(t) \dt \\
%&= \int \int_{\R^{2d}} V_\psi \varphi(x,\xi) V_{\overline{\gamma}} \chi(x,-\xi)\dx \dxi \\
&= \int \int_{\R^{2d}} \left ( \int_{\R^d} \varphi(t) M_{-\xi} T_x \overline{\psi}(t) \dt\right)V_{\overline{\gamma}} \chi(x,-\xi)\dx \dxi \\
&= \int_{\R^d} V^\ast_{\overline \psi}(V_{\overline \gamma}\chi)(t) \varphi(t)\dt \\
&= (\gamma, \psi)_{L^2} \int_{\R^d}\varphi(t) \chi(t) \dt,
\end{align*} 
where the switching of the integrals is permitted because of Corollary \ref{STFT-GSb} and the first part of this proposition.
\end{proof}

In order to show the analogue of Proposition \ref{STFT-gg} for $\widetilde{\mathcal{B}}^{\{M_p\}}_\mathcal{V}(\R^d)$, we need the following technical lemma.
\begin{lemma}\label{admissible-lemma}
Let $\mathcal{V} = (v_n)_{n \in \N}$ be an $A_p$-admissible decreasing weight system. For every $v \in \overline{V}$ there is $\overline{v} \in \overline{V}$ such that
$v(x+y) \leq \overline{v}(x)e^{A(\tau y)}$, $x,y \in \R^d$.
\end{lemma}
\begin{proof} Find a strictly increasing sequence of natural numbers $(n_j)_{j \in \N}$ such that 
$v_{n_{j+1}}(x+y) \leq C_jv_{n_j}(x)e^{A(\tau y)}$, $x,y \in \R^d$, for some $C_j > 0$. Pick $C'_j > 0$ such that $v \leq C'_j v_{n_j}$ for all $j \in \N$. Set $\overline{v} = \inf_{j \in \N} C_jC'_{j+1}v_{n_j} \in \overline{V}$. We have that
$$
v(x+y) \leq \inf_{j \in \N} C'_{j+1}v_{n_{j+1}}(x+y) \leq e^{A(\tau y)}  \inf_{j \in \N} C_jC'_{j+1}v_{n_{j}}(x) =  \overline{v}(x)e^{A(\tau y)}.
$$
\end{proof}

\begin{proposition}\label{STFT-gg-projective} Let $M_p$ and $A_p$ be weight sequences satisfying $(M.1)$ and $(M.2)'$, let $\mathcal{V} = (v_n)_{n \in \N}$ be an $A_p$-admissible decreasing weight system, and let $\psi \in \mathcal{S}^{(M_p)}_{(A_p)}(\R^d)$. Then, the following mappings are continuous: 
$$
V_\psi: \widetilde{\mathcal{B}}^{\{M_p\}}_\mathcal{V}(\R^d) \rightarrow C\overline{V}(\mathcal{V}\otimes\mathcal{V}_{\{M_p\}})(\R^{2d}_{x,\xi}) 
$$
and
$$ 
V^\ast_\psi: C\overline{V}(\mathcal{V}\otimes\mathcal{V}_{\{M_p\}})(\R^{2d}_{x,\xi})  \rightarrow  \widetilde{\mathcal{B}}^{\{M_p\}}_\mathcal{V}(\R^d).
$$
\end{proposition}
\begin{proof}
Let $u \in \overline{V}(\mathcal{V}\otimes\mathcal{V}_{\{M_p\}})$ be arbitrary. By Remark \ref{remark-tensor-weights-2} and \cite[Lemma 4.5$(i)$]{D-V-VEmbeddingUltra16} there is $v \in \overline{V}(\mathcal{V})$ and $r_j \in \mathfrak{R}$ such that $u \leq v \otimes e^{M_{r_j}}$. According to \cite[Lemma 2.3]{Prangoski} there is $r'_j \in \mathfrak{R}$ such that $r'_j \leq r_j$ for $j$ large enough and $r'_{j+1}\leq 2^{j+1}r'_j$ for all $j \in \N$.
Hence the sequence $M_p \prod^p_{j = 0} r'_j$ satisfies $(M.2)'$. Next, by Lemma \ref{admissible-lemma} there is $\overline{v} \in \overline{V}$ such that $v(x+y) \leq \overline{v}(x)e^{A(\tau y)}$ for all  $x,y \in \R^d$. Therefore, Lemma \ref{STFT-test} implies that the mapping
$V_\psi: \mathcal{D}^{M_p, \pi r'_j/\sqrt{d}}_{L^\infty_{\overline{v}}}(\R^d) \rightarrow Cv\otimes e^{M_{r'_j}}(\R^{2d})$
is well-defined and continuous. As the inclusion mapping $Cv\otimes e^{M_{r'_j}}(\R^{2d}) \rightarrow Cu(\R^{2d})$
is continuous, we may conclude that $V_\psi$ is continuous. Similarly, by using Lemma \ref{double-int-test}, one can show that $V^\ast_\psi$ is continuous.
\end{proof}
We are ready to prove our main theorem.
\begin{theorem}\label{main-theorem}
Let $M_p$ and $A_p$ be weight sequences satisfying $(M.1)$ and $(M.2)'$ such that $\mathcal{S}^{(M_p)}_{(A_p)}(\R^d) \neq \{0\}$ and let $\mathcal{V} = (v_n)_{n \in \N}$ be an $A_p$-admissible decreasing weight system satisfying $(V)$. Then, $\mathcal{B}^{\{M_p\}}_\mathcal{V}(\R^d)$ and $\widetilde{\mathcal{B}}^{\{M_p\}}_\mathcal{V}(\R^d)$ coincide topologically.
\end{theorem}
\begin{proof}
By Lemma \ref{equality-sets} it suffices to show that the inclusion mapping $\iota: \widetilde{\mathcal{B}}^{\{M_p\}}_\mathcal{V}(\R^d) \rightarrow \mathcal{B}^{\{M_p\}}_\mathcal{V}(\R^d) $ is continuous. Since $M_p$ satisfies $(M.2)'$, the decreasing weight system $\mathcal{V}_{\{M_p\}}$ satisfies $(S)$ and thus condition $(V)$ (see Remark \ref{remark-S}). Hence Proposition \ref{Bastin} and Remark \ref{remark-tensor-weights-1} imply that $\mathcal{V} \otimes \mathcal{V}_{\{M_p\}}C(\R^{2d}) =  C\overline{V}(\mathcal{V} \otimes \mathcal{V}_{\{M_p\}})(\R^{2d})$ topologically. Choose $\psi, \gamma \in \mathcal{S}^{\{M_p\}}_{\{A_p\}}(\R^d)$ such that $(\gamma, \psi)_{L^2} = 1$. By \eqref{reconstruction-B} the following diagram commutes
\begin{center}
\begin{tikzpicture}
  \matrix (m) [matrix of math nodes, row sep=2em, column sep=2em]
    {\widetilde{\mathcal{B}}^{\{M_p\}}_\mathcal{V}(\R^d) & C\overline{V}(\mathcal{V} \otimes \mathcal{V}_{\{M_p\}})(\R^{2d}) = \mathcal{V} \otimes \mathcal{V}_{\{M_p\}}C(\R^{2d})  \\
    \mathcal{B}^{\{M_p\}}_\mathcal{V}(\R^d) & \mbox{} \\
   };
  { [start chain] \chainin (m-1-1);
\chainin (m-1-2)  [join={node[above,labeled] {V_\psi}}];

}
  { [start chain] \chainin (m-1-1);
\chainin (m-2-1)[join={node[left,labeled] {\iota}}];;
}
  { [start chain] \chainin (m-1-2);
\chainin (m-2-1) [join={node[below,labeled] {V^\ast_\gamma}}];
}
\end{tikzpicture}
\end{center}
Propositions \ref{STFT-gg} and \ref{STFT-gg-projective} imply that $V_\psi$ and $V^\ast_\gamma$ are continuous, whence $\iota$ is also continuous.
\end{proof}
\begin{remark}
Let $M_p$ and $A_p$ be weight sequences satisfying $(M.1)$ and $(M.2)'$ such that $\mathcal{S}^{(M_p)}_{(A_p)}(\R^d) \neq \{0\}$. By applying Theorem \ref{main-theorem} to $\mathcal{V} = \mathcal{V}_{\{A_p\}}$ (and using  \cite[Lemma 4.5$(i)$]{D-V-VEmbeddingUltra16}), we obtain the well known projective description of the classical Gelfand-Shilov space $\mathcal{S}^{\{M_p\}}_{\{A_p\}}(\R^d)$ of Roumieu type \cite[Lemma 4]{Pil94}.
\end{remark}

We end this article by stating an important particular case of Theorem \ref{main-theorem}. Given a positive function $\omega$ on $\R^d$ such that $1/ \omega$ is locally bounded, we define
$$
\mathcal{D}^{\{M_p\}}_{L^\infty_\omega}(\R^d) := \varinjlim_{h \rightarrow 0^+}\mathcal{D}^{M_p,h}_{L^\infty_\omega}(\R^d)
$$
a Hausdorff $(LB)$-space. Furthermore, we write $\widetilde{\mathcal{D}}^{\{M_p\}}_{L^\infty_\omega}(\R^d)$ for the space consisting of all $\varphi \in C^\infty(\R^d)$ such that $\| \varphi \|_{\mathcal{D}^{M_p,r_j}_{L^\infty_\omega}} < \infty$ for all $r_j \in \mathfrak{R}$ and endow it with the topology generated by the system of seminorms $\{ \| \, \cdot \, \|_{\mathcal{D}^{M_p,r_j}_{L^\infty_\omega}} \, : \, r_j \in \mathfrak{R} \}$.
\begin{theorem}\label{main-theorem-cor}
Let $M_p$ and $A_p$ be weight sequences satisfying $(M.1)$ and $(M.2)'$ such that $\mathcal{S}^{(M_p)}_{(A_p)}(\R^d) \neq \{0\}$ and let $\omega$ be a positive measurable function on $\R^d$ such that $\omega(x +y) \leq C\omega(x)e^{A(\tau y)}$, $x,y \in \R^d$, for some $C, \tau > 0$. Then, $\mathcal{D}^{\{M_p\}}_{L^\infty_\omega}(\R^d)$ and $\widetilde{\mathcal{D}}^{\{M_p\}}_{L^\infty_\omega}(\R^d)$ coincide topologically.
\end{theorem}
\begin{proof}
We may assume without loss of generality that $\omega$ is continuous (for otherwise we may replace $\omega$ with the
continuous weight $\tilde \omega = \omega \ast \varphi$, where $\varphi \in \mathcal{D}(\R^d)$ is non-negative and satisfies $\int_{\R^d} \varphi(t) \dt = 1$, since $\mathcal{D}^{\{M_p\}}_{L^\infty_\omega}(\R^d) = \mathcal{D}^{\{M_p\}}_{L^\infty_{\tilde \omega}}(\R^d)$ and $\widetilde{\mathcal{D}}^{\{M_p\}}_{L^\infty_\omega}(\R^d) = \widetilde{\mathcal{D}}^{\{M_p\}}_{L^\infty_{\tilde \omega}}(\R^d)$ topologically). We set $\mathcal{V} = (\omega)_{n \in \N}$ and notice that $\mathcal{V}$ satisfies $(V)$ (see Remark \ref{remark-S}). Hence, by Theorem \ref{main-theorem}, we find that $\mathcal{D}^{\{M_p\}}_{L^\infty_\omega}(\R^d) = \widetilde{\mathcal{B}}^{\{M_p\}}_{\mathcal{V}}(\R^d)$ topologically. The result now follows from the fact that $
\overline{V}(\mathcal{V}) = \{ \lambda\omega \, : \, \lambda > 0 \}$ and, thus, $\widetilde{\mathcal{B}}^{\{M_p\}}_{\mathcal{V}}(\R^d) = \widetilde{\mathcal{D}}^{\{M_p\}}_{L^\infty_\omega}(\R^d)$ topologically.
\end{proof}
\begin{remark}
Theorem \ref{main-theorem-cor} was already shown in \cite[Thm.\ 4.17]{D-P-V} under much more restrictive conditions on $M_p$ and $A_p$ and with a more complicated proof.
\end{remark}

In \cite[Thm.\ 5.9]{Debrouwere-VindasUH2016} we have shown that, if $M_p$ satisfies $(M.1)$ and $(M.2)$ (cf.\ \cite{Komatsu}), the space $\mathcal{S}^{(M_p)}_{(p!)}(\R^d)$ is non-trivial if and only if $(\log p)^p \prec M_p$ (the latter means, as usual, that $M_p^{1/p}/\log p\to\infty$). Hence, we obtain the ensuing corollary.
\begin{corollary}
Let $M_p$ be a weight sequence satisfying $(M.1)$ and $(M.2)$ such that $(\log p)^p \prec M_p$ and let $\omega$ be a positive measurable function on $\R^d$ such that
$\omega(x +y) \leq C\omega(x)e^{\tau |y|}$, $x,y \in \R^d$, for some $C, \tau > 0$. Then, $\mathcal{D}^{\{M_p\}}_{L^\infty_\omega}(\R^d)$ and $\widetilde{\mathcal{D}}^{\{M_p\}}_{L^\infty_\omega}(\R^d)$ coincide topologically.
\end{corollary}


\begin{thebibliography}{99}
\setlength{\itemsep}{0pt}
\bibitem{Bastin} F.~Bastin, \emph{On bornological $C\overline{V}(X)$ spaces}, Arch. Math. \textbf{53} (1989), 394--398.

\bibitem{B-D} C.~A.~ Berenstein, M.~ A.~Dostal, \emph{Analytically uniform spaces and their applications to convolution
equations}, Lecture Notes in Math., vol. 256, Springer-Verlag, Berlin and New York, 1972.

%\bibitem{Bierstedt} K.~D.~Bierstedt, \emph{An introduction to locally convex inductive limits}, in: \emph{Functional
%analysis and its applications} (Nice, 1986), pp. 35--133. World Sci. Publishing, Singapore, 1988.

\bibitem{B-M-S} K.~D.~Bierstedt, R.~Meise, W.~H.~Summers, \emph{A projective description of
weighted inductive limits}, Trans. Amer. Math. Soc. \textbf{272} (1982), 107--160.

\bibitem{D-V-VEmbeddingUltra16} A.~Debrouwere, H.~Vernaeve, J.~Vindas, \emph{Optimal embeddings of ultradistributions into differential algebras}, Monatsh. Math. \textbf{186} (2018), 407--438.

\bibitem{D-VCousin} A.~Debrouwere, J.~Vindas, \emph{Solution to the first Cousin problem for vector-valued quasianalytic functions,} Ann. Mat. Pura Appl. \textbf{196} (2017), 1983--2003. 

\bibitem{Debrouwere-VindasUH2016} A.~Debrouwere, J.~Vindas, \emph{On the non-triviality of certain spaces of analytic functions. Hyperfunctions and ultrahyperfunctions of fast growth,} Rev. R. Acad.  Cienc. Exactas F\'{i}s. Nat. Ser. A. Math. RACSAM \textbf{112} (2018), 473--508.

\bibitem{D-VWeightedind17}  A.~Debrouwere, J.~Vindas, \emph{On weighted inductive limits of spaces of ultradifferentiable functions and their duals}, Math. Nachr., in press, DOI:10.1002/mana.201700395.

\bibitem{D-P-V} P.~Dimovski, B.~Prangoski, J.~Vindas, \emph{On a class of translation-invariant spaces of quasianalytic ultradistributions}, Novi Sad J. Math. \textbf{45} (2015), 143--175.

\bibitem{Ehrenpreis} L.~Ehrenpreis, \emph{Fourier analysis in several complex variables}, Interscience Tracts in Math., no. 17,
Wiley, New York, 1970.

%\bibitem{G-S} I.~M.~Gelfand, G.~E.~Shilov, \emph{Generalized functions. Vol. 2: Spaces of fundamental and generalized functions}, Academic Press, New York-London, 1968.

\bibitem{Grochenig} K.~Gr\"ochenig, \emph{Foundations of time-frequency analysis}, Birkh\"auser Boston, Boston, MA, 2001.

\bibitem{Komatsu} H.~Komatsu, \emph{Ultradistributions I. Structure theorems and a characterization}, J. Fac. Sci. Tokyo Sect. IA Math. \textbf{20} (1973), 25--105.

\bibitem{Komatsu3} H.~Komatsu, \emph{Ultradistributions III. Vector valued ultradistributions and the theory of kernels}, J. Fac. Sci. Univ. Tokyo Sect. IA Math. \textbf{29} (1982), 653--717.

\bibitem{Pil94} S.~Pilipovi\'{c}, \emph{Characterization of bounded sets in spaces of ultradistributions}, Proc. Amer. Math. Soc. \textbf{120} (1994), 1191--1206.

\bibitem{P-P-V}S.~Pilipovi\'{c}, B.~Prangoski, J.~Vindas, \emph{On quasianalytic classes of Gelfand-Shilov type. Parametrix
and convolution}, J. Math. Pures Appl. \textbf{116} (2018), 174--210.

\bibitem{Prangoski} B.~ Prangoski, \emph{Laplace transform in spaces of ultradistributions}, Filomat \textbf{27} (2013), 747--760.

\end{thebibliography}
\end{document}